\documentclass{elsarticle}

\usepackage{amssymb, latexsym, comment, amsmath, amsthm, upgreek, epsf, epsfig, color, verbatim, caption, tabularx}
\usepackage{amsfonts}
\usepackage{booktabs}
\usepackage{siunitx}

\usepackage{array}
\newcolumntype{L}[1]{>{\raggedright\let\newline\\\arraybackslash\hspace{0pt}}m{#1}}
\newcolumntype{C}[1]{>{\centering\let\newline\\\arraybackslash\hspace{0pt}}m{#1}}
\newcolumntype{R}[1]{>{\raggedleft\let\newline\\\arraybackslash\hspace{0pt}}m{#1}}
\makeatletter
\def\ps@pprintTitle{%
\let\@oddhead\@empty
\let\@evenhead\@empty
\def\@oddfoot{\centerline{\thepage}}%
\let\@evenfoot\@oddfoot}
\makeatother

\newtheorem{thm}{Theorem}[section] 
\newtheorem{lemma}[thm]{Lemma}
\newtheorem{prop}[thm]{Proposition}
\newtheorem{cor}[thm]{Corollary}
\newtheorem{defn}[thm]{Definition}
\newtheorem{rem}[thm]{Remark}
\newtheorem{remark}[thm]{Remark}
\newtheorem{example}[thm]{Example}

\begin{document}

\begin{frontmatter}
\title{Mean Curvature and the Wave Invariants of the Basic Spectrum for a Riemannian Foliation}
\author{M. R. Sandoval} 
\address{Department of Mathematics, Trinity College\\ Hartford, CT 06106 United States}
\begin{abstract} 
%Recall that the basic laplacian depends on the mean curvature vector field.
Given a (possibly singular) Riemannian foliation $\mathcal{F}$ with closed leaves on a compact manifold $M$ with an adapted metric, we investigate the wave trace invariants for the basic Laplacian about a non-zero period. We compare them to the wave invariants of the underlying Riemannian orbifold that exists when the leaves in the regular region are identified to points, equipped with the metric that is transverse to the leaves of the foliation. Recalling that the basic Laplacian differs from the underlying orbifold Laplacian by a term that is the mean curvature vector field associated to the foliation, we show that the first wave invariant about any non-zero period $T$ corresponding to geodesics perpendicular to the leaves that all lie entirely in the regular region of $M$ is independent of the mean curvature vector field and depends only on the underlying orbifold structure of the leaf space quotient. Similar results hold on the singular strata whenever the transverse geodesic flow remains confined to the strata defined by the leaf dimension of the foliation. Conversely, closed geodesics that pass through the exceptional leaves depend on the full laplacian, including the leaf-wise metric on the ambient space.% This describes the conditions under which the basic spectrum associated to a particular representation of a leaf space quotient contains geometric information about the singularities of the underlying quotient versus the geometry of the ambient space $M$. 
We also discuss families of representations that yield the same leaf space and basic spectrum. We use this to give conditions under which non-trivial isotropy for orbifold quotients can be detected.

\end{abstract}
\end{frontmatter}

\begin{section}{Introduction} 
%We investigate to what extent the $G$-invariant spectrum captures the underlying geometry of the quotient $M/G$.
\begin{subsection}{Setting}
Let $(M,\mathcal{F})$ be a (possibly) singular Riemannian foliation $\mathcal{F}$ with closed, but not necessarily connected, leaves on a compact manifold, $M$, equipped with an adapted Riemannian metric $g.$ (See \cite{Molino}, Chapter 6, for basic definitions.) Recall that a Riemannian foliation is singular if and only if the leaves do not have constant dimension; otherwise the foliation is said to be regular. We note that if the leaves are not closed, we may always take the closures of the leaves; by \cite{AR2017}, the foliation by leaf closures defines a singular Riemannian foliation on $M$. Hereafter, we will refer to $M$ as the ambient manifold, and denote its dimension by $n.$ The leaf space $Q$ of $(M,\mathcal{F})$ is the space obtained by identifying the leaves of $\mathcal{F}$ to points and with the closed leaf hypothesis, it is a Hausdorff space, whose dimension we denote by $q$; furthermore, it is naturally equipped with a metric, $g_Q$ inherited from the adapted metric on $M$ via the restriction of the adapted metric to the vectors that are orthogonal to the leaves. This metric is independent of the leaf-wise coordinates, by definition of the adapted metric. We denote the leaf space quotient map by $\pi_Q:M \rightarrow Q.$ The quotient $Q$ can be thought of as a potentially singular generalization of a Riemannian manifold. 

Examples of types of singular spaces that arise as leaf spaces in this way include some Alexandrov spaces, space of objects of proper Lie groupoids equipped with a transversally invariant Riemannian metric, Riemannian orbifolds, and (of course) manifolds. An important feature of this setting is that the same leaf space $(Q,g_Q)$ may be obtained as a leaf space quotient in more than one way--that is, very different foliated spaces $(M_1, \mathcal{F}_1)$ and $(M_2,\mathcal{F}_2)$ may produce the same quotient $Q=M_1/\mathcal{F}_1=M_2/\mathcal{F}_2$ with the same induced metric $g_Q$. We refer to a particular Riemannian foliation $(M,\mathcal{F},g)$ that gives rise to $Q$ as a representation of $(Q,g_Q).$ In fact, one can generate infinite families of representations that produce the same $(Q,g_Q)$. (See Section 2 for examples.)

Associated to each representation of $(Q,g_Q)$ given by $(M,\mathcal{F},g)$, we have the following spectral problem. Consider the functions on $M$ that are constant along the leaves of the foliation; these are known as {\it basic} functions, and denoted by $C^\infty_B(M,\mathcal{F}).$ Such functions naturally descend to functions on the leaf space $Q=M/\mathcal{F}$. The restriction of the Laplacian on $M$ to $C^\infty_B(M,\mathcal{F})$ is known as the basic Laplacian, $\Delta_B$, and its spectrum is known as the {\it basic spectrum,} denoted by $spec_B(M,\mathcal{F}).$ It is well-known that the basic spectrum depends to some extent on the choice of the representation $(M,\mathcal{F},g),$ including the leaf-wise metric, and generalizes the Laplace spectrum on functions on $M/\mathcal{F}$, when the leaf space has the structure of a manifold or an orbifold. 

Before proceeding further, we note that in the literature on singular Riemannian foliations that the leaf space has an orbifold structure if and only if it has the property of being {\it infinitesimally polar}, see \cite{LT2010}. This is a property of the quotient, which is independent of the choice of representation; thus, having a Riemannian orbifold structure is a property of the leaf space quotient, in the sense that this property holds for all representations as singular or regular)Riemannian foliations that yield the same quotient and transverse metric. This yields a much broader definition of a Riemannian orbifold than the traditional definition found in the orbifold literature. (See, for example, \cite{ALR2007}) We use this more flexible definition of orbifold here, as it allows for a larger, more interesting version of spectral theory.
\end{subsection}

\begin{subsection}{Families of Spectral Problems}

For a given representation $(M,\mathcal{F}, g)$ of $(Q, g_Q),$ we can consider the region where the leaf dimension is maximal, denoted by $M_{reg}$. On $M_{reg}$, which is open and dense in $M$, we have {\it distinguished coordinates} of the form $(x,y)$ where $y$ denotes the leafwise coordinates and $x$ denotes the transverse coordinates. With respect to these coordinates, the metric can be written as $g(x,y)=g_Q(x)+g_L(x,y)$ where $g_L(x,y)$ denotes the leaf-wise metric and $g_Q(x)$ is the transverse metric; this is due to the nature of the adapted metric on $M$. For every choice of a representation of $(Q,g_Q)$ we obtain a basic spectrum. If we consider all such representations of $(Q,g_Q)$, we obtain a family of spectral problems (and basic spectra), which are all associated to the metric structure of $(Q, g_Q)$. Denote the set of all such representations by $\mathcal{R}(Q,g_Q)$. 

It is natural to ask the following questions: 
\begin{enumerate}
\item[(1)] To what extent does the basic spectrum of a given representation detect the properties of the ambient space $M$ versus the properties of the leaf space $Q$? 
\item[(2)] Can one assign a spectral invariant to $(Q,g_Q)$ by considering the family of basic spectra or a suitably chosen subfamily? A spectral invariant of the quotient $Q$ would have to be associated to the family of all representations in some natural way. For example, it could either depend on a particular canonically chosen representation  (such as the orthonormal frame bundle $M=\mathcal{F}r(\mathcal{O})$ and the $O(q)$ for an orbifold $\mathcal{O}$), or some subfamily, or be independent of all representations.
\item[(3)] For an orbifold quotient $Q=\mathcal{O}$, to what extent can the family 
%(or a well-chosen sub-family) 
of basic spectra be used to distinguish properties of $(Q,g_Q)$ or its singular sets\footnote{There are different notions of singular sets arising from the literature of singular Riemannian foliations versus the corresponding notion from the literature on orbifolds. These distinctions are clarified in Section 2. Both sets of ideas will be explored in this paper.}?
%, and we will be most interested in cases when the two ideas coincide for a particular representation.}? 
%Here, by singular set, we use the term ``singular" in the sense of singular Riemannian foliations; we will use the term ``orbifold singular set" to refer to orbifold singularites. Both ideas will be clarified in this paper.}? 

\end{enumerate}

We will focus primarily on two types of leaf space quotients: when $Q$ is either a manifold or an orbifold (with non-empty orbifold singular set). These situations are distinguished by the fact that the family of representations in these cases will always contain a representation that is a regular Riemannian foliation--namely the orthonormal frame bundle representation which is an example of a {\bf geometric resolution} of the quotient. (See \cite{Lyt2010} for definitions.) In this case, it is always possible to assign spectral invariants to $(Q,g_Q)$ via the basic wave trace invariants in a well-defined way that is independent of the representation. In fact, manifold and orbifold quotients are characterized by the existence of a representation that is a geometric resolution.) Basic wave invariants can be assigned via the geometric resolution for each value of $T$ in the length spectrum of $Q$ which we denote by $Lspec(Q, g_Q)$. For leaf spaces that are more singular than than orbifold quotients, it is not clear that there is a way to assign spectral invariants to $(Q, g_Q)$ due to the lack of a representation that is a geometric resolution.

To shed light on the questions above, we briefly review the definition of singular sets in the sense of singular Riemannian foliations. We refer the reader to \cite{Molino} for standard definitions.

\begin{defn}
Given a representation of $(Q,g_Q)$ denoted by $(M,\mathcal{F}),$ we consider first the stratification of the ambient space $M$.
\begin{itemize}
\item A foliation for which the leaf dimension is constant is a regular Riemannian foliation and the leaf space is, at worst, an orbifold. Thus, if a Riemannian foliation is regular, then its leaf space quotient is an orbifold (possibly with singularities), however the converse is false: an orbifold can also be represented by a singular Riemannian foliation. (See, for example, \cite{GL2016}.)
\item If a Riemannian foliation is singular rather than regular, then let $d(x)$ denote the function that assigns to every $x\in M$ the dimension of the leaf containing $x,$ with values $k,$ ranging over $0\le k_{1}\le k\le k_{N}<n,$ with $k_1$ and $k_N$ denoting the minimal and maximal values for $k$, respectively. (Here we assume there are no dense leaves so that $k_N<n$.) The function $d(x)$ is lower semi-continuous on $M,$ (see \cite{Molino}, Chapter 5), and it defines a stratification of $M$ with each stratum defined as $\Sigma_{k}=d^{-1} (k)$. Thus, each stratum $\Sigma_{k}$ is the (possibly disconnected) union of closed leaves of dimension $k$, and is an embedded manifold that is regularly foliated by the $k$-dimensional leaves, denoted by $(\Sigma_{k},\mathcal{F}_k)$. The images of these sets in the quotient are orbifolds $\mathcal{O}_k=\Sigma_k/\mathcal{F}_k.$ 
\item The lower semi-continuity of $d(x)$ implies that the stratum $M_{reg}$ or $\Sigma_{k_N},$ is an open, dense set and that that for each stratum, $\Sigma_k$, its closure $\overline{\Sigma_{k}}\subset \cup_{\ell\le k}\Sigma_{\ell}.$ The strata which are in the complement of the regular stratum are referred to collectively as the {\it singular strata}, and their leaves are known as {\it exceptional leaves.} 
\begin{comment}
We also note that $\Sigma_{k_1}$ is a compact, totally geodesic submanifold. The corresponding orbifold, $\mathcal{O}_{k_1}=\Sigma_{k_1}/\mathcal{F}_{k_1}$ will be referred to as the minimal orbifold. The projection $\pi_Q,$ restricted to $M_{reg}$ is a Riemannian submersion, and its image in $Q$ is, at worst, an orbifold of dimension $q=n-k_N$ that we denote by $\mathcal{O}.$
\item Since the foliation of each stratum is regular, the leaf spaces $\mathcal{O}_k=\Sigma_k/\mathcal{F}_k$ are at worst Riemannian orbifolds, equipped with the transverse part of induced metric on $(\Sigma_k, \mathcal{F}_k)$, which is bundle-like for this foliation. 
\end{comment}

\end{itemize}
\end{defn}

%The stratifications vary across the family of representations.
\begin{remark}\label{stratvary} We emphasize that the singular sets depend on the choice of the representation. For a given $Q,$ the stratification is {\bf not necessarily} constant across the family of representations, unless $Q$ is a manifold or an orbifold that does not admit singular representations. (Such an orbifold would probably have to be a very good orbifold, given that the representations are global quotients.) In fact, some representations are better than others depending on whether or not one wishes capture the orbifold singularities in the regular stratum or the singular strata. In Section 2.1, we illustrate this phenomenon for various examples. The main idea here is to use the fact that some representations are preferred over others for answering particular questions.
%For the purposes of addressing the first question, we might prefer representations for which the stratification by leaf dimension matches up better with the stratification of the orbifold by singular sets. 
For the purpose of using the wave invariants to address the first two questions, we will be interested in choosing representations whose strata are well-situated with respect to the dynamics of the geodesic flow across $M$, as we shall see below. To address the third question, we will use a representation of an orbifold that allows us to extract the length spectrum of closed geodesics in $Q$ whose length is affected by passing through an orbifold singularity. We will describe this phenomenon below and prove the corresponding results in the Section 4.

\end{remark}

First, we describe how two representations in a family can yield the same basic spectrum. 
\begin{remark}\label{isospectral}
To that end, we make the following remarks.

\begin{enumerate}
\item The basic spectrum is affected by leaf-wise data--both the leaf dimension and the leaf volume. The role of the leaf volumes in spectral questions involving the basic spectrum is evident from the well-known formula which is valid on $M_{reg}$ relating the basic Laplacian on the ambient space $M$, $\Delta_B$, with the Laplacian on $\mathcal{O}$, denoted by $\Delta_\mathcal{O}$:
\begin{equation}\label{e:compformula}
\Delta_B f=\Delta_{\mathcal{O}} f-H_*(f),
\end{equation}
where $f$ is a basic function on $M$, which descends to a function on the leaf space, also denoted by $f,$ and similarly $H$ also descends to a vector field on $\mathcal{O},$ denoted by $H_*.$ 

\item It has been shown in \cite{AR2016b} that  the mean curvature vector field, $H$, can be expressed as in terms of the leaf volume as follows
\begin{equation}\label{meancurvature}
H(x)=\nabla \log\Bigl(\frac{1}{Lvol(x)}\Bigr),
\end{equation}
where $Lvol(x)$ denotes the volume of the leaf through the point $x\in M_{reg}$. Thus, $H$ is a conservative vector field of a smooth function on $M_{reg}$, and its dual 1-form (the mean curvature form) is exact. Furthermore, it is also known that when $H$ is basic, the basic spectrum is a subset of the Laplace spectrum on $M$, equipped with the adapted metric.

\item Next, we quote the following result from \cite{AS2020} which gives criteria for when two representations of the same leaf space have the same basic spectra, given a suitable map, called an SRF isometry, (see \cite{AS2020} for definitions) between the leaf spaces:

{\bf Theorem:} [Adelstein-Sandoval, \cite{AS2020}]\label{2019result} 
Let $(M_1, \mathcal{F}_1)$ and $(M_2, \mathcal{F}_2)$ be two singular Riemannian foliations with mean curvature vector fields $H_1$ and $H_2$, respectively, and let $\varphi \colon M_1/\mathcal{F}_1 \to M_2/\mathcal{F}_2$ be a smooth SRF isometry satisfying the following two conditions:  (1) $H_1$ and $H_2$ are basic vector fields, and  (2) $d\varphi(H_{1*})=H_{2*}$ on an open, dense set. Then the leaf spaces are {\bf basic isospectral}, i.e.~$spec_B(M_1,\mathcal{F}_1)=spec_B(M_2,\mathcal{F}_2).$

(See \cite{AS2020} for the definition of SRF isometry.) A singular Riemannian foliation that satisfies condition (1) is said to be a {\it generalized isoparametric} singular Riemannian foliation. We will now assume henceforward that any singular Riemannian foliation that represents $Q$ is a generalized isoparametric singular Riemannian foliation.

%\item There exist infinite families of representations $(M,\mathcal{F})$ that are basic isospectral, as we shall see in Section 2.1.

\item The formula \eqref{e:compformula} and the above result imply that the basic spectrum is precisely the orbifold spectrum when $H$ is identically zero, but otherwise is likely to yield a different spectrum altogether. (See Example 1, \cite{AS2020}). 

\end{enumerate}
\end{remark}

Our first result demonstrates that the set  $\mathcal{R}(Q,g_Q)$ is infinite and contains an infinite basic isospectral family.  It is proved in the Section 4 and will have applications to both Questions (2) and (3) above:
\begin{thm}\label{infinitefamilies}
For any leaf space $(Q,g_Q)$, where $Q=M/\mathcal{F}$ with mean curvature vector field $H$, there exists an infinite subfamily of representations $\mathcal{R}_\mu(Q,g_Q, H)\subset \mathcal{R}(Q,g_Q)$ indexed by $\mu\in \mathbb{R}^+=(0,\infty)$ that are all basic isospectral and whose ambient space is $M$ but with the leaf-wise metric rescaled by constant $\mu$ on $M_{reg}$.
\end{thm}

We will show that the basic spectra contain geometric information about the orbifold structure of the underlying leaf space for relatively closed curves that avoid exceptional leaves, even when $H$ is not identically zero. Furthermore, we describe under what conditions the leading wave invariant of the basic spectrum detects properties of the underlying orbifold $\mathcal{O}\subseteq Q$ versus those of the ambient space $M$. 

The approach taken in this paper is to use the wave trace invariants about a non-zero period $T$ associated to certain relatively closed geodesics that are orthogonal to the the leaves of the foliation; such curves are {\it relatively closed with respect to the foliation} (for definitions and more detail, see Section 2). We will denote the set of such relatively closed geodesics by $\mathcal{RT}(M,\mathcal{F}),$ which is equal to $Lspec(Q,g_Q)$. 

\begin{comment}
For orbifold quotients, the lengths of geodesics that are relatively closed but not smoothly closed are precisely those in the set $\mathcal{RT}(M,\mathcal{F})\setminus Lspec^\perp(M,\mathcal{F}),$ and it is these curves that reflect the presence of orbifold singularities. We similarly denote the set of periods of smoothly closed geodesics in $M$ that {\bf are not} orthogonal to the leaves of $(M,\mathcal{F})$ as $Lspec^{non\perp}(M,\mathcal{F})$; it includes leaf-wise closed geodesics and geodesics that are transverse to the leaves but not orthogonal to the leaves
\end{comment}

Note also: the set of relatively closed geodesics in $M$ of length $T$ is stratified by leaf dimension as well.

%maybe move this after the wave trace results addressing first question.

\end{subsection}
\begin{subsection}{Detecting Quotient Space Versus Ambient Space Geometry}

Our next set of results address Question (1). In particular, we show that the first basic wave invariant is independent of the mean curvature vector field, and, thus, does not detect the variability of leaf volumes in the regular region.

\begin{thm}\label{main1} 
Let $T\in\mathcal{RT}(M,\mathcal{F})$ be a period of a relatively closed geodesic that satisfies the following two conditions: (1) the geodesic projects to a geodesic in $M/\mathcal{F}$ that is contained in the image of the regular region, and (2) the set of closed geodesics in $Q$ of length $T$ has maximal dimension (in $Q$). Then, under standard clean intersection hypothesis, the basic wave trace expansion localized about $T$ has degree and leading coefficient $\sigma_0(T)$ given by the corresponding wave invariant associated to the orbifold $\mathcal{O}=M_{reg}/\mathcal{F}.$
\end{thm}

\begin{comment}
\begin{cor}\label{cor:1}
Under the hypothesis of the previous theorem, the leading term in the basic wave trace wave invariant cannot detect the variability of leaf volumes over the regular region. 
\end{cor}
\end{comment}

We also have the following, which says that the corresponding wave invariant is independent of the representation:

\begin{cor}\label{cor:2}
Suppose $(M',\mathcal{F}')$ is another singular Riemannian foliation representing the leaf space of $(M,\mathcal{F})$ that has the same transverse metric $g_Q$, and for which the set of relatively closed geodesic curves of length $T$ are contained in the regular regions for both representations. Then under the same hypotheses as the previous theorem on the curves of length $T$, the two representations will yield the same leading term for the basic wave trace. This will be true even if $M'_{reg}\not=M_{reg}.$ Here we denote $M_{reg}'/\mathcal{F}'=\mathcal{O}'.$ Thus, the leading term in the basic wave trace for this $T$ is independent of the representation of the orbifold $\mathcal{O}\cap\mathcal{O}'$.
\end{cor}

\begin{rem}
Having answered the question of whether or not the first wave invariant of the basic spectrum detects the mean curvature vector field, it is natural to ask if it detects changes in leaf dimension. Perhaps surprisingly in light of Remark 1.3, item 3, the answer is possibly yes, in the sense that the first wave invariant of singularities corresponding to closed geodesics that pass through exceptional leaves (and thus experience a change in leaf dimension) may depend on the leaf-wise metric depending on the whether or not the projection of the  component of fixed point set of the time $T$-geodesic flow that is largest in $Q$ contains a geodesic that passes through an exceptional leaf, as we shall see in the theorem below.
\end{rem}

More generally, for a given representation $(M,\mathcal{F})$ with $Q=M/\mathcal{F}$, we can decompose $Lspec(Q,g_Q)$ into the set of relatively closed periods that correspond to geodesics that avoid exceptional leaves, $\mathcal{RT}_{orb}(M,\mathcal{F})$, and its complement--those for which there is a geodesic curve of length $T$ that passes through exceptional leaves, $\mathcal{RT}_{exc}(M,\mathcal{F}).$
Then we have the following disjoint decomposition of the length spectrum of $(Q, g_Q)$:

\begin{equation}
Lspec(Q,g_Q)=\mathcal{RT}_{orb}(M,\mathcal{F}) \cup \mathcal{RT}_{exc}(M, \mathcal{F})
\end{equation}

\begin{thm}\label{newtrace}
The basic wave trace $Trace(U_B(t))$ then admits an expansion near $T \in Lspec(Q,g_Q)$ as a sum of lagrangian distributions on $\mathbb{R}$ that whose degree is equal to $-\frac{1}{4}-\frac{e^T}{2}$ where $e^T$ the maximum dimension of the projection of the fixed point set of the time $T$ geodesic flow in $M$ to $Q$. Furthermore, the basic wave trace decomposes as follows:
\begin{equation}
 TrU_B(t)=I(t)+II(t)
\end{equation}
where the leading term of $I(t)$ depends on the underlying quotient orbifolds $\mathcal{O}_k$, and the leading term of $II(t)$ may depend on the representation and the ambient space. $I(t)$ is non-zero only when there exist contributions of maximum dimension in $Q$ that arise from geodesics that avoid exceptional leaves. 
\end{thm}

\begin{remark}
We note the following:
\begin{enumerate}
\item We will defer the precise statement of this trace formula, which is rather lengthy and technical, and its proof until Section 3.
\item For manifold quotients, all the representations are regular (See Section 2.1), and thus, $\mathcal{RT}_{exc}(M, \mathcal{F})\cup \mathcal{RT}_b(M,\mathcal{F})=\emptyset,$ and the first wave invariant is independent of the mean curvature form and the leaf volume for each $T\in Lspec(Q,g_Q)$. 

\item Similarly for orbifold quotients for which all the representation in $\mathcal{R}(Q,g_Q)$ are regular Riemannian foliations, the leading wave invariant is independent of the mean curvature term for every $T$. The first wave invariant is constant for all representations in $\mathcal{R}(Q,g_Q)$. Hence such quotients are are spectrally distinct from quotients for which $\mathcal{RT}_{exc}(M, \mathcal{F})\not=\emptyset.$

\item If all the geodesics of lengths $T\in \mathcal{RT}_{exc}(M, \mathcal{F})$ pass through exceptional leaves  then the basic wave trace can be localized into parts that depend on either the underlying quotient orbifolds or the ambient space. Such representations are preferable for answering Question (1) above.
%\item For a leaf space quotient that is non a manifold or an orbifold, it is known via \cite{Lyt2010} that there is no regular representation that preserves $g_Q$. Therefore, therefore every representation of such a $Q$ is singular. 
\end{enumerate}
\end{remark}

%In terms of Question XX, we seek a representation t

\end{subsection}

\begin{subsection}{Detecting Orbifold Singularities}

With regard to Question (3) above, we now restrict our attention to leaf spaces that are either orbifolds or manifolds. The goal here is to find orbifold singularities. Here, we keep in mind that there are two aspects to the basic wave trace invariants which arise from the expansion of the wave trace as a Lagrangian distribution about its singularities: (1) the location of those singularities which are contained in the length spectrum of $Q,$ and (2) the particular form of the expansion itself. Detecting the presence of orbifold singularities in this case is less about the form of the expansion than about the length spectrum. Notionally, one way to detect orbifold singularities is to find a closed curve geodesic in $Q$ that arises from a curve in the ambient space $M$ that is relatively closed, but not smoothly closed. Here we will take $M$ to be $\mathcal{F}r(Q)$ and the foliation to be that arising from the orbits of the usual $O(q)$ action. A closed geodesic curve that passes through an orbifold singularity, say $\overline{p}\in Q$ will be one whose image in $\mathcal{F}r(Q)$ is relatively closed with respect to the foliation but is not smoothly closed. If $\pi(p)=\overline{p}$, then this curve will be closed in $Q$ because of the action of some element of the isotropy group at $p,$ denoted by $G_p.$ If such geodesics exist for $Q$ then we can detect an orbifold singularity at $\overline{p}$. For quotients $Q$ with no orbifold singularities, such curves do not exist. The problem here is how to distinguish the smoothly closed curves in the frame bundle from the relatively closed curves in $\mathcal{RT}(\mathcal{F}r(Q),\mathcal{F}).$ Let $Lspec^\perp(\mathcal{F}r(Q),g)$ denote the set of lengths of smoothly closed geodesic curves in $\mathcal{F}r(Q)$ that are orthogonal to the orbits of the $O(q)$ action with respect to an $O(q)$-invariant metric $g$.

Then we have the following as an application of the above result and the general approach of looking at families of representations, we will show the following as a consequence of Theorem \ref{infinitefamilies}.

\begin{thm}\label{orbifoldsingularity}
Given a Riemannian orbifold $(Q, g_Q)$ with $O(q)$ frame bundle $\mathcal{F}r(\mathcal{O}),$ then in the notation above,%Let $T\in \mathcal{RT}(\mathcal{F}r(Q),O(q))$.
\begin{equation}\label{invariantperiods}
Lspec^\perp(\mathcal{F}r(Q),O(q))=\bigcap_{\mu\in (0,\infty)}Lspec(\mathcal{F}r(Q),g_\mu)
\end{equation}
 and thus is independent of the choice of representation in the subfamily $\mathcal{R}_\mu(\mathcal{O}, g_\mathcal{O}, H=0)$ and depends only on $Q$. Furthermore, if there exists a length $T\in \mathcal{RT}(\mathcal{F}r(Q),O(q))$ that is not in $Lspec^\perp(\mathcal{F}r(Q),O(q)),$ then $Q$ is an orbifold with non-empty singular set, and there exists a closed geodesic curve of length $T$ that contains an orbifold singularity, $\overline{p}$.
 \end{thm}
 
 We can also detect the order of an element in $G_p$ as follows:
 
 \begin{cor}\label{corollaryorbifoldsingularity}
 Under the hypotheses of the previous theorem, there exists a smallest integer $k>1$ such that for some $T\in Lspec^\perp(\mathcal{F}r(\mathcal{O}), O(q))$, but $jT\not\in Lspec^\perp(\mathcal{F}r(\mathcal{O}), O(q)),$ for any $1<j<k,$ and thus $G_p$ contains an element of order $k$.
\end{cor}

%need no overlap--add to hypotheses

\begin{comment}
\begin{cor}
If a manifold $M$ and an orbifold $\mathcal{O}$ with non-empty singularity set (in the sense of orbifolds) are isospectral, then there cannot be closed geodesics that pass through the singular set that are not smoothly closed. Otherwise, the 
\end{cor}
\end{comment}
\end{subsection}

\begin{subsection}{General Remarks and Organization}
We use the wave invariants for the basic spectrum as opposed to the basic heat invariants (or, equivalently, the wave invariants at the zero singularity), because they have some advantages in this context. One advantage is that the wave invariants about a non-zero $T$ are generally neither too local nor too global in nature. One problem with the heat invariants (and the wave trace invariants about $T=0$ is that they involve functions (the local heat invariants) that must be integrated over the entire manifold, and thus over the singular strata, which can be problematic. As K. Richardson has shown in \cite{KR2}, the local heat invariants, for example, are not always integrable. On the other hand, the wave invariants over non-zero $T$ involve integrals over the sets in $S^*M$ fixed by the time $T$ hamiltonian flow of the transverse metric, and objects defined thereon. In this paper, we will focus on certain curves whose fixed point sets are well-behaved with respect to the singular sets. This approach takes advantage of the fact that the dynamics of the hamiltonian flow of the transverse metric behave rather nicely with respect to the singular strata; see Section 2 for the particulars.

One disadvantage to the wave invariants at non-zero $T$ is that one cannot be guaranteed that they necessarily appear in the asymptotic expansion; it is always possible that the singularities associated to different curves of the same length could cancel each other out, although it is known in many specific instances that some they do not cancel, \cite{Sutton2016}, and generically, they do not cancel. Thus, the results described here that involve using the wave trace expansion localized to a particular non-zero period $T$ can only be used to distinguish spectra of different leaf spaces, {\it in principle,} assuming the coefficients in the asymptotics do not cancel.

\begin{comment}
We recall that wave trace asymptotics arise from the taking the trace of the basic wave operator $U_B(t,x,y)$ and, under standard hypotheses involving the clean-ness of the fixed point sets (see Section 2 for precise statements), have the form
\begin{equation}\label{generalform1}
Trace\bigl(U_B(t,x,y))\bigr)=\sum_{T\in \mathcal{RT}(M,\mathcal{F})}\nu_{T}(t),
\end{equation}
where $I^{-\frac{1}{4}-\ell}(\mathbb{R},\Gamma^T,\mathbb{R})$ denote the degree $-\frac{1}{4}-\ell$-th degree lagrangian distributions over the half lines $\Gamma^T=\{(t,\tau)\,|\,\tau<0\}$ and each $\nu_T$ has the form
\begin{equation}\label{generalform2}
\nu_{T}(t)=e^{\frac{i\pi m_T}{4}}\sum_{j=0}^\infty\sigma_j(T)(t-T+i0)^{-\ell-\frac{1}{2}-j}\,mod\,C^\infty(\mathbb{R}),
\end{equation}
where $m_T$ is the Maslov index of the fixed point set under the time $T$ hamiltonian flow.
\end{comment}

% move this

The results here are concerned almost entirely with the first wave invariant. The lower order terms in the wave trace expansion are quite complex even in the case of a manifold (see for example, \cite{Z1999}). It is entirely possible that the variation of leaf volumes via the mean curvature form can be detected in the lower order terms.

This paper represents a continuation of the author's previous work on the basic spectrum involving some special cases of the basic wave trace for a singular Riemannian foliation in \cite{San3}, \cite{San4}, \cite{San2013}, and is related to \cite{AS2017}, \cite{AS2020}. Similar work with the heat invariants has been done by \cite{PaRi}, \cite{KR2}, and \cite{Ri2}. The asymptotics of the heat kernel for orbifolds has been studied in \cite{DGGW}, where it was shown that the heat invariants of the singular sets of orbifolds can detect information about the singular sets. The orbifold results here are also related to the results of \cite{SU2011}.

The paper proceeds as follows. In Section \ref{background}, we illustrate the existence of many representations of a leaf space quotient $(Q, g_Q)$ with examples, and the proof of Theorem \ref{infinitefamilies}. Section 3 contains a  review the necessary background for the main tool for proving Theorem \ref{main1}, Corollary \ref{cor:2}, and \ref{newtrace}, and the statement and proof of the the generalization of the wave trace for singular Riemannian foliations with disconnected leaves. In Section 4, we prove Theorem \ref{orbifoldsingularity}, and the related corollary. 
%A clarification to the author's previous work on the basic wave trace in \cite{San2013}, which is needed for the proof of the generalization of the wave trace, is included in an Appendix.

%Should we include applications to group actions, as an application? What are these? Possibly talk about the various representations: regular, Coxeter, Searle Reductions.

\end{subsection}
\end{section}

\begin{section}{Background}\label{background}
\begin{subsection}{Examples of Representations of the Same Leaf Space}
Here we provide a series of examples to illustrate how leaf space quotients can be represented in different ways as leaf spaces of Riemannian folliations. The examples that follow include foliations with both connected and disconnected leaves, as well as examples with minimal and non-minimal leaves. We begin with the following definitions, which can be found in the literature.

\begin{defn}\label{disconnected}
A {\bf (singular) Riemannian foliation with disconnected leave}s is the triple $(M,\mathcal{F}, \Gamma)$ where $(M,\mathcal{F})$ is a singular Riemannian foliation with connected leaves and $\Gamma$ is a discrete group of isometries of $Q=M/\mathcal{F}$ that is used to ``glue" together the leaves of $\mathcal{F}$ to form new disconnected leaves as follows. The group $\Gamma$ extends naturally to an action on the leaves of $\mathcal{F}$ so that the orbits of $\Gamma$ form the disconnected leaves of the foliation. In other words, if $L$ is a leaf for $\mathcal{F}$ then we define the corresponding leaf of the disconnected foliation to be $L':=\Gamma\cdot L.$  
\end{defn}

In the case of disconnected leaves, the mean curvature vector field and leaf space isometries can be defined as before. But, as we shall illustrate with examples, we will need to distinguish between principal leaves in a foliation with connected leaves versus the principal leaves in the disconnected case. 

\begin{defn}{disconnected2}
The {\bf principal leaves of a disconnected singular Riemannian foliation} are the sets $L'$ formed by taking the disjoint union of the components (which we denote we by $L$) that have maximal dimension and trivial leaf holonomy, and have the additional property that if there exists an $f\in \Gamma$ such that $f(L)=L$ for some component $L$ of $L'$, then $f$ must be the identity in $\Gamma$.
\end{defn}

\begin{example}
If $Q$ of dimension $q$ is a closed compact manifold $X$ (without boundary) then we always have the following representations.
\begin{enumerate}
\item We can take the ambient space $M$ to be the $O(q)$ frame bundle with the fibres of the bundle to be the leaves, and the usual metric will be bundle-like for the regular Riemannian foliation. Note that the leaves are disconnected with $\Gamma=\mathcal{Z}_2\cong O(q)/SO(q),$ but if $X$ is orientable, then the $SO(q)$ frame bundle with fibres as the connected leaves forms another regular Riemannian foliation with adapted metric.
\item In a similar way, any bundle over $X$ with compact fibres can serve as the ambient space and forms a regular Riemannian foliation with appropriate metric. 
\item At the other extreme, we can consider any finite covering map over $X,$ of degree $k$. The cover $\tilde{X}$ is the ambient space of  a regular Riemannian foliation with totally disconnected leaves and $\Gamma=\mathbb{Z}_k$ acting trivially.
\item The trivial foliation of a manifold by points.
%\item The Hopf Fibrations for spheres of appropriate dimensions.
\item Any simple foliation on a connected (or disconnected) space with $p:M\rightarrow X$.
\item Any Riemannian submersion $p:M\rightarrow X$ with connected fibres, with manifold base $X$ (so with trivial holonomy).
\item Homogeneous spaces, where $X=G/N$ is a Lie group where $N\lhd G$ are Lie groups and $G$ compact.
\end{enumerate}

Many of the above examples are {\bf minimal representations} by which we mean that the corresponding mean vector field is identically zero. Given any minimal representation one can also create a non-minimal representation by replacing the leaf-wise metric with one that has variable leaf volumes.
\end{example}

In the above cases, all the examples have only principal leaves, and thus yield only regular Riemannian foliations. 

For representations with disconnected leaves, we note that on the connected portions of the leaves, the leaf holonomy is trivial and the action by $\Gamma$ satisfied the condition above. We note that in these disconnected examples, the action of $\Gamma$ induces an action on the subspace of the tangent space that is normal to the leaves, and that this action is trivial. If it were otherwise, the leaf space would have an orbifold singularity. In all cases above, the entire ambient space forms the regular stratum, and all the leaves are principal.

Next we consider orbifold quotients, and we note that a regular Riemannian foliation can result in an orbifold quotient.

\begin{example}
If $Q$ is an orbifold $\mathcal{O}$ then we have the following examples of representations.

\begin{enumerate}
\item We can take the ambient space $M$ to again be the $O(q)$ frame bundle, in the nonorientable case and the $SO(q)$ frame bundle otherwise. These representations are regular, but they may have non-principal leaves with $\Gamma=\mathbb{Z}_2$ acting with fixed components in the non-orientable case for these leaves. Here, the regular region is the entire ambient space, and the mean curvature vector field is identically zero. We note that in the non-orientable case the leaf holonomy induced by the action of $\Gamma$ will not be trivial.
\item If $\mathcal{O}$ has as its underlying space $|\mathcal{O}|,$ a manifold with totally geodesic boundary, and the usual $Z_2$ action obtained after doubling $|\mathcal{O}|$ over its boundary, then this yields another example of a regular foliation. Here, the non-principle leaves are those over the boundary points, and principal leaves are on on the interior.
\item As an example of the above, the interval $\mathcal{O}=[-\pi, \pi]$ is an orbifold with reflectors at the boundary points corresponds to an orbifold as traditionally defined in the orbifold literature. If we represent this same quotient as the leaf space that arises by identifying the lines of constant latitude on a standard 2-sphere (here have the action of $S^1$ acting by rotations through the poles) then we have a singular Riemannian foliation with two strata: the regular region is the sphere with the poles deleted, and the singular stratum consists of the poles. Here, the mean curvature vector field arises from the taking the gradient of the function that is equal to the inverse of the circumferences of the latitude circles.
\item Homogeneous spaces with orbifold quotients, $\mathcal{O}=G/H$ where are  $G$ and $H$ Lie groups. If the orbits are either principal or exceptional, then the representation will be regular. This example will have one stratum, but possibly variable orbit volumes. If there are singular orbits, then the quotient will be a singular quotient with the singular orbits comprising the singular strata.
\item Any singular Riemannian foliation that is infinitesimally polar (\cite{Lyt2010}). 
\item We always have the foliation by points, once more.
\item For any quotient given by the leaf space of a homogenous singular Riemannian foliation, one can generate new representations by reducing the action in some way or by fattening up the leaves due to the method of K. Richardson in \cite{Ri2}.
\end{enumerate}
\end{example}
%By a result due to A. Lytchak, an orbifold is characterized by the condition that every representation is {\it infinitesimally polar}. (Look up and Cite) Further one can define an orbifold point $x\in M$ as one for which the infinitesimal foliation of $N_xL_x$ is polar. If in addition, the leaf holonomy is trivial, then $x$ is a manifold point.

Orbifolds are characterized by the property that one can find a geometric resolution. In the above, the frame bundle representation is always a geometric resolution. An orbifold may also admit representations with non-zero mean curvature vector fields.

%Note: in the traditional literature on orbifolds, many of the above examples would not be considered to be orbifolds at all. However, from the singular Riemannian foliations perspective, any SRF that is infinitesimally polar is an orbifold, and this property is a property of the quotient and not the representation.

\begin{remark}\label{dynamicsremark}
We note that orbifolds come with a natural stratification arising from isotropy type, however, this stratification does not necessarily coincide with the stratification of a (singular) Riemannian foliation. For the purposes of studying the orbifold singularities via the basic wave trace invariants, we must seek representations that best capture the dynamics of the geodesic flow, if possible. Often this representation is one for which the orbifold singular set of interest coincides with the image under $\pi_Q$ of one of the orbifolds $O_k$ and the dynamics of the geodesic flow are such that the relatively closed curves do not pass through exceptional leaves.
\end{remark}

\begin{remark}
While manifolds only admit regular representations, it is not known if there exist orbifolds that only admit representations as a regular quotient. We conjecture that if such an orbifold exists it must be a very good, orientable orbifold with isolated singularities because otherwise the principle isotropy reduction of the usual $O(q)$ action on the frame bundle would yield singular strata, (See \cite{GS2000} for the principle isotropy reduction.) If it were the case that there are no such orbifolds that are not manifolds, then the property of $Q$ admitting only regular Riemannian foliation representations would distinguish manifolds from orbifolds in a way that might prove useful.
\end{remark}

If $Q$ with a fixed metric $g_Q$ is not a Riemannian orbifold or manifold, then it cannot be represented as the leaf space of a regular Riemannan foliation. This is implied by the work A. Lytchak of \cite{Lyt2010}. The quotient $(Q, g_Q)$ must therefore have a stratification with singular strata, and therefore it must also have a non-zero mean curvature vector field, since due to the main result of \cite{MiWol2006}, there are no singular Riemannian foliations with minimal leaves--they must be regular. 
%In many cases, the basic wave trace invariants can be used to distinguish non-orbifold quotients from regular quotients, but that is a subject for a follow-up paper.

%[Note: cite prior result about isospectral orbifold and non-orbifold singularities.]

\begin{remark}
We note that in \cite{AS2017}, it was shown that there exist representations of orbifold and non-orbifold quotients that are basic isospectral using a generalization of Sunada's method.
\end{remark}

We refer to Theorems 6.2  and 6.7 of \cite{ABT2013}, which together characterize the property of a point in $Q$ being belonging to an orbifold for a leaf space quotient $Q=M/\mathcal{F}.$ We note that one can define manifold, orbifold, and non-orbifold points as follows.
\begin{defn}
Let $Q=M/\mathcal{F}$ be a leaf space quotient with $\pi: M\rightarrow Q$ and $\pi(x)=\bar{x}$. 

\begin{enumerate}
\item We define $Q_{orb}$ to be the set $$\{\bar{x}\in Q\,|\, \forall \, (M,\mathcal{F}) \text{ with } Q=M/\mathcal{F} \text{ where the infinitesimal foliation at } x\text{ is polar}.\}$$ See Definition 2.5 of \cite{ABT2013}, for definition of the infinitesimal foliation.

\item Similarly we define $Q_{man}$ to be the set 
\begin{equation*}
\begin{split}
\{\bar{x}\in Q\,|\,\forall \, (M,\mathcal{F}) \text{ with } Q=M/\mathcal{F} \text{ where the infinitesimal foliation at } &x\text{ is polar }\\ \quad\text{ and the leaf is principal}\}
\end{split}
\end{equation*}
\item Finally, $Q_{non-orb}=(Q_{orb})^c$.
\end{enumerate}
\end{defn}

%\begin{remark}
%By the result of\cite{somepaperbyLytchak}, the codimension of $Q_{non-orb}$ is at least two. We also know that $\pi_Q^{-1}(Q_{non-orb})\cap M_{reg}=\emptyset.$ It follows from the proof  of theorem of \cite{AS2020} stated in item 3 of Remark \ref{isospectral} that the presence of non-orbifold points is not audible to the basic spectrum.
%\end{remark}
We conclude this section with the proof of Theorem \ref{infinitefamilies}.

\begin{proof}[Proof of Theorem \ref{infinitefamilies}.]

First suppose $Q$ is given a regular Riemannian foliation with leaf dimension equal to $k=k_{N}.$ Then $Q=M/\mathcal{F}$ with $H=-\nabla\log(Lvol(x)),$ which may be zero. Let $g(x,y)=g_Q(x)+g_L(x,y)$ be the (bundle-like) metric on $M$, where $x$ and $y$ denote the transverse and leaf-wise coordinates, respectively. Then for each $\mu\in \mathbb{R}^+$ one can alter the metric so that over the regular region we replace the leaf-wise metric $g_L(x,y)$ by $\rho g_L(x,y)$ been rescaled conformally by some positive conformal factor $\rho\in \mathbb{R}^+$ so that the rescaled leaf volume is $\mu Lvol(x),$ where $\mu= (\rho^{k_{N}})^{1/2}.$ Let $g'_\mu(x,y)$ denote the new metric. One derives immediately from \eqref{meancurvature}, that the mean curvature vector field $H'_\mu$ for $(M,\mathcal{F}, g'_\mu)$ is equal to the original mean curvature form $H.$ By \cite{AS2020}, every representation given by such a rescaling is basic isospectral to the original one.

Now suppose $Q$ is not a regular Riemannian foliation\footnote{We will not need this case for the proof of subsequent results.}, again with maximum leaf dimension equal to $k_N$. Then we can perform the same rescaling of the leaf-wise metric over the regular region, and extend the entire metric on $M$ so that it is a smooth metric that yields the same transverse metric $g_Q$. Note that the new metric obtained in this way only needs to be preserve $g_Q$, not the leaf-wise metric, either over the regular region or the singular strata. One derives immediately from \eqref{meancurvature}, that the mean curvature vector field $H'_\mu$ for $(M,\mathcal{F}, g'_\mu)$ is equal to the original mean curvature form $H$. Thus, by \cite{AS2020}, each representation $(M,\mathcal{F}, g'_\mu)$ is basic isospectral to the original one. 
\end{proof}

\begin{remark}\label{remarkforsectionfour}
We note the following for use in Section 4.
\begin{enumerate}
\item From the above proof, observe that the length spectrum of $(Q,g_Q)$ is the same for all representation in $\mathcal{R}(Q,g_Q, H)$. However, the length spectrum of $M$ has been altered for the closed geodesics that correspond to non-transverse curves in $M$. 
\item For the purposes of calculating the wave invariants of the basic spectrum, we can choose any representation from $\mathcal{R}(Q,g_Q, H)$ and some may be more suited to a given calculation of the localized wave trace than others for a particular period $T$.
\end{enumerate}
\end{remark}
\end{subsection}

\end{section}

\begin{section}{Proof of the Basic Wave Trace Formula Results}

In this section we state and prove the results behind Theorem \ref{newtrace} which implies Theorem \ref{main1} and Corollary \ref{cor:2}. The basic wave trace results follow from applying microlocal analytic techniques to analyze the singularities of the basic wave operator.  These techniques naturally take place in the cotangent bundle of the ambient space $M.$ These singularities propagate along certain hamiltonian curves in the cotangent bundle which project to geodesics on $M$ which are orthogonal to the leaves of the foliation. When the trace of the wave kernel is taken, the result is a distribution on the real line whose singularities are located at the lengths of certain geodesics which are relatively closed with respect to the foliation. This yields the Poisson relation for the basic wave trace. The foliated geometric structure lifts to a stratification of a portion of the cotangent bundle of  $M$, found in \cite{San2013}. This will yield a stratified singular configuration space, which will be needed to understand the interplay of this stratification with the dynamics of the geodesic flow. 

We begin by describing the stratified singular configuration space. Next we describe the properties of the hamiltonian flow of the transverse metric and describe the relatively closed curves for singular Riemannian foliations with possibly disconnected leaves. We then define the basic wave operator and calculate the Poisson relation. We can then state and prove the basic wave trace formula. Finally, by localizing to specific periods that either stay confined to the various strata or pass through exceptional leaves, we show Theorem \ref{newtrace} and its corollaries.

\begin{subsection}{The Stratification of the Singular Configuration Space}\label{configspace}

The foliation and the associated stratification of $M$ induces a foliation and a corresponding stratification of a subset of $T^*M,$ which will serve as the configuration space and arises as follows. Let $\pi:T^*M\rightarrow M$ denote the usual map on the cotangent bundle $(T^*M,\,\omega)$, equipped with its usual symplectic form.  We will use the following notation: for any distribution $V$ of $TM,$ let $V^0$ denote the subset of $T^*M$ defined as follows
\begin{equation*}
V^0=\{\xi_x\in T^*M\,|\,\xi_x(v_x)=0\,\,\forall v_x\in V_x\}.
\end{equation*}

Let $T\mathcal{F}\subset TM$ denote the distribution of variable dimension that is defined by the tangent spaces to the leaves. Then the space $(T\mathcal{F})^0$ admits a singular foliation defined as follows. 

\begin{defn}\label{d:singconfigstrata}
The foliation of each stratum induces a foliation $N\Sigma_k=(T\Sigma_k)^\perp$ for each leaf dimension $k\ge 0$ as follows:

\begin{equation}
(N\Sigma_k)^0=\{\xi_x\in T^*M\,|\,\xi_x(v_x)\,\forall v_x\in N\Sigma_k\}.
\end{equation}
Then we define
\begin{equation}
\Sigma^*_k=(N\Sigma_k)^0\cap (T\mathcal{F})^0
\end{equation}
for $0<k_1\le k\le k_N$. Each such stratum is foliated by $k$-dimensional leaves arising from the null foliation for the symplectic form $\omega$ restricted appropriately to the stratum $\Sigma^*_k$ for each $k.$
We can also define a stratum (foliated by points) over the points that lie in $(T\mathcal{F})^0$ but are not in the above strata $\Sigma^*_k$ for $k>0$.

If $k_1=0$ we define
\begin{equation}\label{e:sigmastarsupzero}
\Sigma_0^*=\Bigl(\bigcup_{\ell} \bigl(N\Sigma_{\ell})^0\bigr)^c\cup(N\Sigma_0)^0\Bigr)\cap (T\mathcal{F})^0,
\end{equation}
for $k_1\le \ell\le k_N$, and if $k_1>0$ we note that $N\Sigma_0=\emptyset$ in the above. We consider this stratum to be foliated trivially by its points. This yields a stratification of $(T\mathcal{F})^0$, which we shall call the {\it singular configuration space}, as in \cite{San2013}. 
%working here 3/2
Later, we will wish to distinguish certain points in $\Sigma^*_0$, that we define to be {\it exceptional}. Thus, let
\begin{equation}\label{e:sigmaexc}
\Sigma^*_e=\Bigl(\bigcup_{\ell} \bigl((N\Sigma^*_{\ell})^0\bigr)^c\Bigr)\cap (T\mathcal{F})^0,
\end{equation}
for $k_1\le \ell\le k_N.$

Let $(\Sigma_k^*, \widetilde{T\mathcal{F}_k})$ denote the foliation of each stratum of the singular configuration space for each $k$, and let \begin{equation}\label{singconfigspace}
(T\mathcal{F})^0:=\bigcup_{k=0}^{k_N} \Sigma_k^*,
\end{equation}
be the stratified singular configuration space.
\end{defn}

Associated to this stratified space is the groupoid, $\mathcal{G}^*(\mathcal{F}),$ with special properties, given in the subsequent proposition. This proposition is a generalization of the one proved in Proposition 1 of \cite{San2013}, Section 3; its proof is identical to that of Proposition 1 in \cite{San2013}.
    
\begin{prop}\label{t:groupoidprop2}
There exists a groupoid $\mathcal{G}^*(\mathcal{F})$ associated to $(T\mathcal{F})^0$ such that the orbits of $\mathcal{G}^*(\mathcal{F})$ are $k$-dimensional leaves that foliate each stratum $\Sigma^*_k.$ For each $k\ge 0,$ each of the $k$-dimensional leaves of $\Sigma^*_k$ is a union of $k$-dimensional leaves induced by the null foliation on $(N\Sigma_k)^0$ when $k>0$ or the trivial foliation by points if $k=0$.
\end{prop} 

\end{subsection}

\begin{subsection}{Properties of the Hamiltonian flow of the transverse metric on the stratified configuration space}

Next, we show that the hamiltonian flow of the transverse metric behaves well with respect to the stratification described above. Let $\mathcal{H}(\xi_x)=\|\xi_x\|_x$ denote the usual hamiltonian function which is determined by the adapted metric on $M$. As a consequence of the metric being adapted to the singular Riemannian foliation, the metric is bundle-like for each foliated stratum $(\Sigma_k,\mathcal{F}_k).$ This means that the metric $(g_{ij})$ and the induced metric on  $T^*M,$ $(g^{ij}),$ can be written locally in distinguished coordinates with respect to the foliation $(\boldsymbol{x},\boldsymbol{y})$ on  $(\Sigma_k,\mathcal{F}_k)$ where $\boldsymbol{x}$ and $\boldsymbol{y}$ denote the transverse and leaf-wise coordinates, respectively, as follows:
\begin{equation}\label{bundlelikemetric}
g_{ij}(\boldsymbol{x},\boldsymbol{y})=g_{ij}^T(\boldsymbol{x})+g_{ij}^L(\boldsymbol{x},\boldsymbol{y})
\end{equation}
where $g_{ij}^T$ and $g_{ij}^L$ denote the transverse and leaf-wise metrics, respectively.
As in \cite{San3}, the hamiltonian flow of $\mathcal{H}$ restricts to $(T\mathcal{F})^0=\{\mathcal{H}(\xi_x)=\mathcal{H}_{\mathcal{F}^\perp}(\xi_x)\}$, where $\mathcal{H}(\xi_x)=\mathcal{H}_{\mathcal{F}^\perp}(\xi_x),$ the part of the hamiltonian that arises from the transverse part of the bundle-like metric on each $\Sigma_k\subset M$.  This transverse flow on $\Sigma^*_k$, which we denote by $\Phi^t(\xi_x),$ is particularly well-behaved with respect to the singular configuration space $(T\mathcal{F})^0$ and its stratification. In fact, as was shown in Section 3 of \cite{San2013}, the transverse flow restricts to the singular configuration space, and sends leaves to leaves, with respect to all types of leaves in Proposition \ref{t:groupoidprop2}.  We can generalize from Section 3 of \cite{San2013} the following proposition:

%---------------------------------------------------------------------------------------------------------------------------------------------------------------

%
\begin{prop}\label{t:flowstrat2}
The transverse hamiltonian flow of $H$ restricts to the singular configuration space.  In other words, if $\xi_x\in (T\mathcal{F})^0$ (with $\xi_x\in\Sigma^*_k$) then the hamiltonian vector field $\Xi_{\mathcal{H}}(\xi_x)\in T\Sigma^*_k$ and 
\begin{equation}\label{e:flowrestricts}
\{\Phi^t(\xi_x)\,|\,t\in\mathbb{R}\}\subset (T\mathcal{F})^0.
\end{equation} 
Furthermore,  if $k>0$ then either 
\begin{equation}\label{e:singstrat1}
\{\Phi^t(\xi_x)\,|\,t\in\mathbb{R}\}\subset \Sigma_k^*
\end{equation}
or
\begin{equation}\label{e:singstrat2}
\{\Phi^t(\xi_x)\,|\,t\in\mathbb{R}\}\subset \Sigma_k^*\cup\Sigma_e^*.
\end{equation}
If $k=0$ then 
\begin{equation}\label{e:singstrat3}
\{\Phi^t(\xi_x)\,|\,t\in\mathbb{R}\}\subset \Sigma_0^*\setminus\Sigma_e^*.
\end{equation}Otherwise, $\Xi_{\mathcal{H}}(\xi_x)$ is not tangent to $\Sigma^*_e$. 
\end{prop}
It follows that such a transverse hamiltonian curve can be tangent to only one of the $\Sigma^*_k$.

\medskip

Equipped with these results we can define the notion of a curve being {\it relatively closed with respect to the foliation}:

\begin{defn}\label{d:relclosed2}
We say that an arc of the hamiltonian flow $\Phi^t$ through $\xi_x$ will be said to be {\bf relatively closed with respect to the singular Riemannian foliation} $\mathcal{F}$ with period $T$ if we have the following for each of the following cases:

\begin{enumerate}
\item[Case 1: ]$k_N=0$, disconnected leaves. There exists an element $f\in \Gamma$ such that $df^*(\xi_x)=\Phi^T(\xi_x)$.
\item[Case 2: ]$k_N>0$, leaves connected. 
Recall first that sliding along the leaves along a path $\alpha$ that is contained in a leaf defines a local diffeomorphism of the transversal space which we denote by $h_\alpha$. This local diffeomorphism only depends on the homomtopy class of $\alpha;$ this defines holonomy for the leaf. An arc of the transverse hamiltonian flow $\Phi^t$ through $\xi_x$ will be said to be {\it relatively closed with respect to the singular Riemannian foliation} $\mathcal{F}$ with period $T$ if for endpoints $\xi_x$ and $\eta_y=\Phi^T(\xi_x),$ there is a homotopy class $[\alpha]$ of a curve $\alpha$ wholly contained in the leaf with endpoints $x$ and $y$, such that $dh^*_{\alpha^{-1}}(\xi_x)=\Phi^T(\xi_x).$\ This is equivalent to the existence of a groupoid element of the form $[\xi_x,\Phi^T(\xi_x),dh^*_{\alpha^{-1}}]\in \mathcal{G}^*(\mathcal{F}).$ Let $\mathcal{RT}(M,\mathcal{F})$ denote the set of lengths of hamiltonian arcs that are relatively closed with respect to the foliation of $(M,\mathcal{F}).$
There exists a groupoid element $\gamma$ in $\mathcal{G}^*(\mathcal{F})$ such that $\gamma(\xi_x)=\Phi^T(\xi_x).$  
 More concisely, we say that the curve is relatively closed if the endpoints $\xi_x$ and $\Phi^T(\xi_x)$ belong to the same leaf.

\item[Case 3: ] $k_N>0$, leaves disconnected. The endpoints $\xi_x$ and $\Phi^T(\xi_x)$ belong to the same disconnected leaf.
\end{enumerate}
\end{defn}
\begin{defn}\label{defnRT}
Let $\mathcal{RT}(M,\mathcal{F})$ denote the set of lengths of hamiltonian arcs that are relatively closed with respect to the foliation of $(M,\mathcal{F}),$ as in the previous definition.
\end{defn}

\begin{remark}
Note that for each stratum $\Sigma_k$ we have $\mathcal{RT}(\Sigma_k,\mathcal{F}_k)\subset \mathcal{RT}(M,\mathcal{F})$ and  $Lspec(\mathcal{O}_k)\subset \mathcal{RT}(M,\mathcal{F}).$
\end{remark}

\begin{remark}
Let $\gamma(t,x)=\pi(\Phi^t(x,\xi)),$ where $\pi:T^*M\rightarrow M;$ these are geodesics in $M$. Let $\Xi_{\mathcal{H}}$ denote the hamiltonian vector field of the transverse bundle-like metric. Then, in local distinguished coordinates, it is easily seen that $\gamma'(0,x)=(d\pi)_{\xi_x}(\Xi_{\mathcal{H}})\perp T_x\mathcal{F}.$ It is known that if $\gamma(t,x)$ is a geodesic passing through $x$ that is perpendicular to the leaf at one point, then $\gamma(t,x)$ remains perpendicular to all the leaves that it meets, (see Chapter 6 of \cite{Molino}). Thus, the projections of such relatively closed hamiltonian curves are geodesic arcs that are orthogonal to all the leaves through which the geodesic passes.

\end{remark}

\end{subsection}

\begin{subsection}{The generalization of the basic wave trace for a singular Riemannian foliation}
\begin{comment}
\begin{proof}
We divide the proof into three cases: the case when $k_N=0$ ($Q$ is a good orbifold and the leaves are points or disjoint unions of points.), when $k_N>0$ and the leaves are connected, and finally the case when $k_N>0$ and the leaves have disconnected components.

When $k_N=0$ the groupoid is simply given by the action of the group $\Gamma$ as follows. Note here $Q=M/\Gamma.$ The action of $\Gamma$ induces an action on $T^*M$ in the usual way. We take $\mathcal{G}^*(\mathcal{F})$ to be the groupoid given by the induced action on $T^*M$.

When we have a singular Riemannian foliation with connected leaves, then the result follows from Corollary 10 of \cite{San2013}. 

Finally, if $k_N>0$ and we can have disconnected leaves, we take $\mathcal{G}^*(\mathcal{F})$ to be given as in the previous case, and we use the induced action of $\Gamma$ to define the disconnected leaves.
\end{proof}
\end{comment}
%
% move the above to the appropriate position 3/15/2022
%

Next we recall how the basic wave operator and its Schwartz kernel are defined. 

\begin{defn}\label{basicprojector}
\begin{enumerate}
\item
Recall from \cite{San2013} that there is a basic projector $P$ from the space of smooth functions on $M$ to the space of basic functions on $M$, given by the averaging of functions over the leaves: If $L_x\Sigma_k$ is the leaf containing $x\in M$ then for leaves that are not totally disconnected we have
\begin{equation}\label{e:projector}
P(f)(x)=\frac{1}{Lvol_k(x)}\int_{L_x}f(y)dvol({L}_y),
\end{equation}
where $f$ a function on $M$, and $Lvol(x)$ is the volume of the leaf that contains $x.$ We allow here for the leaves to have disconnected components. In what follows, it will be convenient to indicate the dimension of the leaves and their volumes. Let $Lvol_k(x)$ denote the volume of the leaf of dimension $k$ that contains $x$.

On the other hand, if there are leaves that are totally disconnected with cardinality $|L_x|$ then
\begin{equation}\label{e:projector2}
P(f)(x)=\frac{1}{|L_x|}\sum_{y\in L_x}f(y),
\end{equation}
where $f$ a function on $M$, and $L_x$ is the leaf that contains $x.$
\item From the above, the operator $P,$ is essentially the push-forward $(\pi_*)|_{\Sigma_k}$ multiplied by the function $\frac{1}{Lvol_k(x)}$. Its symbol, as an operator, is $\frac{1}{Lvol_k(x)}\sigma((\pi_*)|_{\Sigma_k}).$

The case of totally disconnected leaves would be handled similarly.

\item The basic wave trace is the trace of the basic wave operator, denoted by $U_B(t)$ and its Schwartz Kernel is given by 
\begin{equation}\label{e:P}
U_B(t,x,y)=P_xP_yU(t,x,y)=\Sigma_{j=1}^\infty e^{-it\sqrt{\lambda_j^B}}e_j(x)e_j(y),
\end{equation}
where $U(t,x,y)$ is the wave kernel of the ordinary Laplacian on $M$, $0\le \lambda_1^B\le \lambda_2^B\le \cdots$ are the eigenvalues of the basic Laplacian, $e_j$ denote the corresponding basic eigenfunctions, and $P$ denotes the projector onto the space of basic functions. The subscripts $x$ and $y$ on $P$ indicate which of the first two factors of $M\times M\times \mathbb{R}$ is being acted upon.
\end{enumerate}
\end{defn}
As suggested by the previous paragraph, we will be interested in composing the basic projector $P$ with the wave operator. Consequently, we will need the following lemma.

\begin{lemma}\label{canonicalP}
The stratified canonical relation $\mathcal{C}$ of the basic projector $P$ is given by 
\begin{equation}
\{(\eta_y,\xi_x)\,|\,\xi_x\in (T\mathcal{F})^0, \exists\, \boldsymbol{\alpha}\in \mathcal{G}^*(\mathcal{F}), \eta_y=\boldsymbol{\alpha}\cdot \xi_x\}.
\end{equation}
\end{lemma}

\begin{proof}
The proof is quite similar to that of the corresponding proposition (\cite{San2013}, Proposition 27) of the original proof, except that over points in $\Sigma_0^*$ the basic projector does not propagate singularities, and thus its canonical relation is that of the identity operator. In other words, the underlying foliation is the trivial foliation by points. Thus, 
\begin{equation}
\mathcal{C}\cap (\Sigma^*_0\times \Sigma^*_0)=\{(\xi_x,\xi_x)\,|\,\xi_x\in \Sigma^*_0\}.
\end{equation}
\end{proof}

%[insert something about the relative length spectrum]

We then have the following Poisson relation for the wave front set of the basic wave kernel:
\begin{thm}\label{t:sojourntimes2}
In the notation previously established, 
\begin{equation}\label{e:invariants}
WF\bigl(Trace(U_B(t,x,y) \bigr)\subset\{(T,\tau)\,|\,T\in\mathcal{RT}(M,\mathcal{F}),\,\tau<0\}.
\end{equation}

\end{thm}
The proof of this result can be taken nearly verbatim from the prior result in \cite{San2013}, where the result was proved for a singular Riemannian foliation that arose as the closure of a regular Riemannian foliation or from \cite{Sangroupwave}. However, now we drop this assumption and simply assume that $(M,\mathcal{F})$ is a singular Riemannian foliation with closed leaves that does not necessarily arise as the closure of a regular Riemannian foliation. We note the the leaves may be disconnected.

\begin{remark}\label{ztk}
In light of Proposition \ref{t:flowstrat2}, we observe that each $Z^T$ can be decomposed with respect to the stratification and establish the following notation: 

\begin{equation}
Z^T=\bigcup_{k=0}^{k_N} Z^T_k,
\end{equation}
where for each $k$, $Z^T_k$ is the union of all the relatively closed curves of length $T$. (For some values of $k$, $Z^T_k$ may be empty.) Let $S(Z^T_k)=S^*(M)\cap Z^T_k;$ this set is saturated by the $k$ dimensional leaves since the transverse metric is constant on the leaves of $Z^T_k$. 
%Let $\kappa(T)$ be the smallest positive integer such that $Z^T_{\kappa(T)}\not=\emptyset.$
\end{remark}
\begin{remark}
We note that for a given $T$ the set $Z_T$ may contain both curves whose saturation contains exceptional leaves and curves that remain confined to the associated strata. This will be important in proving Theorem \ref{newtrace}.
\end{remark}

The following lemma from \cite{San2013} still holds in this case, but we alter its statement here to make the contribution of the function $\frac{1}{Lvol_k(x)}$ that arises from the symbol of $P$ explicit from Definition \ref{basicprojector}. 

\begin{lemma}\label{l:density}
There exists a smooth canonical density, $d\mu_{Z^T_k}$, on each component of the relative fixed point set $Z^T_k.$ Furthermore, $d\mu_{Z^T_k}=\frac{1}{Lvol_k(x)}d\pi^*\chi_{\mathcal{F}_k}\otimes d\mu'_{Z^T_k}$ where $d\mu'_{Z^T_k}$ is a canonical measure on the orbifold $N(Z^T_k)= S(Z^T_k)/\widetilde{T\mathcal{F}}$ and $\chi_{\mathcal{F}_k}$ is the characteristic measure on $(\Sigma_k,\mathcal{F}_k).$
\end{lemma}
\begin{rem}\label{r:leafvolume}
It will be useful in what follows to note the observation in \cite{T1}, p. 37, that $\chi_{\mathcal{F}_k}$ is none other than $dLvol_k(x),$ the volume for the $k$-dimensional leaves of $(\Sigma_k, \mathcal{F}_k),$ and so $\frac{1}{Lvol_k(x)}d\pi^*\chi_{\mathcal{F}_k}$ is a probability measure on each leaf in the sense that each leaf has volume one with respect to this measure. Let $d\eta_k$ denote the measure $\frac{1}{Lvol_k(x)}d\pi^*\chi_{\mathcal{F}_k.}$
\end{rem}

\begin{defn}\label{d:clean}
For each foliated stratum of the configuration space $(\Sigma^*_k, \widetilde{T\mathcal{F}_k} ),$  let $\widetilde{T\mathcal{F}_k}$ denote the transverse distribution.  Then the local diffeomorphisms $dh_{\alpha}$ of Definition \ref{d:relclosed2} define a lifted holonomy action on the leaves of the foliation $(\Sigma^*_k,\widetilde{T\mathcal{F}_k}),$ which we denote by $d\widetilde{h}_{(\alpha,\xi_x)},$ as follows. For any $\boldsymbol{\alpha}\in \mathcal{G}^*(\mathcal{F})$ with $\alpha(0)=x$, $\alpha(1)=y$, and $\xi_x$ with $\eta_y=(dh_{\alpha^{-1}})^*(\xi_x)$, then
\begin{equation}
d\widetilde{h}_{(\alpha,\xi_x)}:\widetilde{N_{\xi_x}\mathcal{F}_k}\rightarrow\widetilde{N_{\eta_y}\mathcal{F}_k}.
\end{equation}
We say that the relatively closed set of curves $Z^T_k$ is {\bf clean} if (1) $Z^T_k$ is a smooth submanifold of $(T\mathcal{F})^0;$ and (2) for every $\xi_x\in Z^T_k$ with $\eta_y=(dh_{\alpha}^{-1})^*\xi_x=\Phi^T(\xi_x)$ then $d\Phi^T_{\xi_x}(T_{\xi_x}Z^T_k)=T_{\eta_y}Z^T_k$ for $\alpha\in \mathcal{G}(\mathcal{F})$ with $\alpha(0)=x$ and $\alpha(1)=y$. The condition that $\eta_y=(dh_{\alpha}^{-1})^*\xi_x$ implies that for all $\xi_x\in Z^T_k$
\begin{equation}\label{e:cleanone}
d\Phi^T_{\xi_x}(\widetilde{N_{\xi_x}\mathcal{F}_k})=d\widetilde{h}_{(\alpha,\xi_x)}(\widetilde{N_{\xi_x}\mathcal{F}_k})=\widetilde{N_{\eta_y}\mathcal{F}_k}.
\end{equation}
\end{defn}
We will need the following definition:
\begin{defn}
Let $S$ be a saturated set with respect to the singular foliation of the stratified singular configuration space $(T\mathcal{F})^0.$ 
The {\bf quotient dimension} of $S$ is defined as follows
\begin{equation}\label{qdim}
qdim(S)=max_k (dim(S\cap \Sigma^*_k)-k).
\end{equation}
\end{defn}
\begin{remark}
We establish the following notation. Let $\Gamma^T=\{(T,\tau)\,|\,\tau<0\}$ denote the ray over $T\in \mathcal{RT}(M,\mathcal{F})$.  Henceforward, we assume that the set $Z^T_k$ of relative fixed points of the hamiltonian flow $\Phi^T$ on $(T\mathcal{F})^0$ is clean for all $T\in  \mathcal{RT}(M,\mathcal{F})$ in the sense of Definition \ref{d:clean}. Let $Z^T_{max}$ denote the union of the components of $Z^T$ for which the quotient dimension is largest. Henceforward, let $Z^T_k$ denote the strata that make up $Z^T_{max}$ to avoid excessive subscripts. Note: the set $S(Z^T_{max})$ is also saturated. The relatively closed hamiltonian arcs of a given length $T$ of maximal quotient dimension decompose as previously noted in Remark \ref{ztk} into conic submanifolds $Z^T_k$ whose connected components are finite in number. Let $\varepsilon^{T}:=dim(S(Z^{T}_k)),$ and let $e^T$ be the corresponding quotient dimension so that $\varepsilon^{T}=k+e^{T}.$
\end{remark}

Then we have the following in the notation above:

\begin{thm}\label{t:newsirftrace}
Near $t=T\in \mathcal{RT}(M,\mathcal{F})$ and assuming that the $Z^T_k$ satisfy Definition \ref{d:clean}, then, in the above notation,  
\begin{equation}
Trace\bigl(U_B(t,x,y))\bigr)=\sum_{T\in \mathcal{RT}(M,\mathcal{F})}\nu_{T}(t),
\end{equation}
where $\nu_T\in I^{-1/4-e_T/2-j}(\mathbb{R},\Gamma^T,\mathbb{R}).$ Furthermore, $\nu_{T}$ has an expansion of the form
\begin{equation}\label{e:expansionnewsirftrace}
\nu_{T}(t)=\sum_{k=0}^{k_N} e^{\frac{i\pi m^T_k}{4}}\sum_{j=0}^\infty\sigma_j(T,k)(t-T+i0)^{-\frac{e_T+1}{2}-j}\,mod\,C^\infty(\mathbb{R}),
\end{equation}
where $m^T_k$ is the Maslov index of $Z^{T}_{k}.$ 
If $Z^T_k\cap\Sigma^*_k=\emptyset$ then we set $\sigma_j(T,k)=0$.

In particular, the leading term $\sigma_0(T,k)$ is
\begin{equation}\label{e:newsigmazero}
\sigma_0(T,k)=\Bigl[\int_{S(Z^T_{k})}\sigma(U) d\mu_{Z^T_{k}}\Bigr]\,\tau^{\frac{e_T-1}{2}}\sqrt{d\tau}.
\end{equation}

Furthermore, if $Z^T_k$ contains no curves that pass through exceptional leaves, then the leading term $\sigma_0(T,k)$ is
\begin{equation}\label{e:newsigmazero2}
\sigma_0(T,k)=\Bigl[\int_{N(Z^T_{k})}\sigma(U) d\mu'_{Z^T_k}\Bigr]\,\tau^{\frac{e_T-1}{2}}\sqrt{d\tau}.
\end{equation}

\end{thm}

%The statement above is identical to the corrected statement in the Appendix with $p=0$. 

\begin{remark}\label{clarifications}
The statement of this theorem is slightly different than that found in \cite{San2013} to clarify several points.
\begin{enumerate} 
\item This generalization of that theorem clarifies the order of the singularity. In the previous version, the dimension of the fixed point set was maximized and the leaf dimension was minimized, and this tacitly assumes that the contributions come from only one stratum, but this is not necessarily the case. The contribution of highest degree comes from the component(s) of the fixed point set of largest quotient dimension, and there is no ostensible reason why different strata of the fixed point set could not contribute subsets of the same quotient dimension. For this reason we take the sum over all the strata that might contribute to the fixed point set of highest quotient dimension. The order of the singularity was expressed in previous versions in terms of $\varepsilon^T=e^T+k$ in the previous remark, but this degree is always reduced by the dimension of the leaves, $k$, and so it makes sense to express the degree in terms of $e^T$ instead.
\item We emphasize the role of the leaf volumes in light of Lemma \ref{l:density} and Remark \ref{r:leafvolume}, which are closely related to the mean curvature vector fields.
\item In the event that there are no curves passing through exceptional leaves, the last part of the theorem will allow us to calculate the contribution $\sigma_0(T,k)$ solely in terms of the orbifold $\mathcal{O}_k,$ as we shall see in the next proposition following the proof of the above theorem.
\end{enumerate}
\end{remark}
\begin{comment}
\begin{remark}
We make note of the following changes in notation from the proof in \cite{San2013}. Here we denote the singular configuration space by $(T\mathcal{F})^0$ instead of $(T\overline{\mathcal{F}})^0$ in the original proof. We note that lift of a leaf of the form $L_x=\{x\}$ to $(T\mathcal{F})^0$ is just $\{\xi_x\}$ for $\xi_x\in T^*M\setminus\{0\}$. Thus, the globally defined space corresponding to the underlying foliation that was formerly denoted by $(T\mathcal{F})^0$ is really just $T^*M\setminus\{0\}$ as in the discussion preceding Definition \ref{d:clean}. Even if $\Sigma_0=\emptyset$, the stratified configuration space will still have leaves consisting of singleton covectors over the exceptional points.
\end{remark}
\end{comment}
%statement
\begin{proof}{Sketch of Proof.}
We first consider the case of totally disconnected leaves. This is essentially the original result in \cite{DG1975}, or in the case of orbifolds, the result of E. Stanhope and A. Uribe in \cite{SU2011}. Generalizing the proof to allow for such disconnected leaves boils down to expanding the definition of the relatively closed geodesics and accounting for the presence of totally disconnected leaves with Remark \ref{basicprojector}. 

For the case of leaves with connected components, the proof of the generalization can be taken almost verbatim from the proof of the corresponding result in \cite{San2013}, aside from the clarifications noted in Remark \ref{clarifications}. The prior result in \cite{San2013}, proved the result for a singular Riemannian foliation that is the closure of a regular Riemannian foliation whose leaf dimension is $p$. 

We make note of the following changes in notation from the proof in \cite{San2013}. Here we denote the singular configuration space by $(T\mathcal{F})^0$ instead of $(T\overline{\mathcal{F}})^0$ in the original proof. We note that lift of a leaf of the form $L_x=\{x\}$ to $(T\mathcal{F})^0$ is just $\{\xi_x\}$ for $\xi_x\in T^*M\setminus\{0\}$. Thus, the globally defined space corresponding to the underlying foliation that was formerly denoted by $(T\mathcal{F})^0$ is really just $T^*M\setminus\{0\}$ as in the discussion preceding Definition \ref{d:clean}. Even if $\Sigma_0=\emptyset$, the stratified configuration space will still have leaves consisting of singleton covectors over the exceptional points.

We make note only of the minor changes required to generalize the result. The main modification consists of setting the leaf dimension of the underlying foliation to $p=0$ in the prior proof and by accounting for the fact that the canonical relation for the basic projector is slightly different in this case, as described in Lemma \ref{canonicalP}. The results up to \eqref{e:newsigmazero} follow immediately.

To show equation \eqref{e:newsigmazero2}, we simply integrate over the leafwise portion via the density $|d\eta_k|,$ which we may due without difficulty since every geodesic in the set $Z^T$ is contained entirely in the regular foliation $\Sigma^*_k,$ by hypothesis.

%working here on May 5

\end{proof}

The rest of the result follows from the following proposition.

\begin{prop}\label{orbifoldcoefficient}
If $T$ is such that the $Z^T_{max}$ does not contain curves that pass through exceptional leaves then for each $k$ with $Z^T_k\not=\emptyset$
\eqref{e:newsigmazero2} simplifies to 
\begin{equation}
\sigma_0(T,k)=\sigma_0(T,\mathcal{O}_k),
\end{equation}
where $\sigma_0(T,\mathcal{O}_k)$ denotes the first wave invariant of the orbifold $\mathcal{O}_k$ for $T\in \mathcal{RT}(\Sigma_k,\mathcal{F}_k)$ which is equal to $Lspec(\mathcal{O}_k).$
\end{prop}
%The first wave invariant...

%I will have to add some stuff here.
\begin{proof}
We proceed by cases. First consider the curves that correspond to geodesics over the regular region.
Suppose $T$ is a relatively closed period in $\mathcal{RT}(M,\mathcal{F})$ such that $Z^T_k=\emptyset$ for all $k<k_N$.  
By equations \eqref{e:expansionnewsirftrace}, and \eqref{e:newsigmazero2}
the expansion of the basic wave trace in a neighborhood of $T$ is of the form
\begin{equation}
\nu_T(t)=\sum_{j=0}^\infty \sigma_j(T)(t-T+i0)^{-\frac{e_T+1}{2}-j}
\end{equation}
and
\begin{equation}
\sigma_0(T,k)=\Bigl[\int_{N(Z^T_{k_N})}\sigma(U) d\mu'_{Z^T_{k_N}}\Bigr]\,\tau^{\frac{e_T-1}{2}}\sqrt{d\tau},
\end{equation}
where $\sigma(U)$ denotes the symbol of the (ordinary) wave kernel which depends on the entire metric, including the leaf-wise part.

We will demonstrate the result by calculating the first wave invariants for an appropriate $T$ on both $M$ and $\mathcal{O},$ and comparing them. To this end, denote $\sigma_0(T,(M,\mathcal{F}))$ and $\sigma_0(T,\mathcal{O})$ the first wave invariants for the basic spectrum for $(M, \mathcal{F})$ and $\mathcal{O},$ respectively. Let $\sigma(U)$ denote the symbol of the wave operator on $M$, and let $\sigma(U_\mathcal{O})$ denote the symbol of the wave operator on $\mathcal{O}.$

Now proceeding as in \cite{DG1975}, we recall from Theorem 4.5 of \cite{DG1975} that $\sigma(U)$ is given by the expression
\begin{equation}\label{firstwave}
\frac{1}{2\pi} \int_{S(Z^T_{k_{N}})}exp\Bigl(\frac{-i}{T} \int _0^T sub(\Delta_B)(\Phi^s(x,\xi))\,ds\Bigr)\, d\mu_{Z^T_{k_{N}}} .
\end{equation}

Next we calculate of the subprincipal symbol of the basic Laplacian locally on $\Sigma_{k_N}.$ Recall from \cite{DG1975} the following calculus of subprincipal symbols: if $Q$ is an operator of degree $m$  with principal symbol $q(x, \xi),$ for coordinates $(x,\xi)\in T^*M$ then
\begin{equation}
sub(Q)(x,\xi)=q_{m-1}(x,\xi)-\frac{1}{2i}\sum_{j=1}^n\frac{\partial^2 q}{\partial x_j\partial\xi_j}(x,\xi),
\end{equation}
and also:
\begin{equation}
sub(Q^\alpha)=\alpha p^{\alpha-1} sub(Q)
\end{equation}
Thus if $Q=\Delta_B=\Delta_{\mathcal{O}}-H$ (locally on $\Sigma_{k_N}$) where $H$ is the mean curvature vector field, then locally, 
\begin{equation}
sub(\Delta_\mathcal{O}-H)(x,\xi)=-h(x,\xi)+sub(\Delta_\mathcal{O})(x,\xi)=-h(x,\xi),
\end{equation}
since it is well-known that $sub(\Delta_\mathcal{O})=0$.
Hence,
\begin{equation}
sub([\Delta_\mathcal{O}-H]^{1/2})(x,\xi)=\frac{1}{2}[\|\xi\|_x]^{-1/2}sub(\Delta_\mathcal{O}-H)(x,\xi).
\end{equation}
On the fixed point sets, which are in the cosphere bundle, $\|\xi\|_x=1$, it follows that when restricted to those sets, we have:
\begin{equation}
sub([\Delta_\mathcal{O}-H]^{1/2})(x,\xi)=-\frac{1}{2}h(x,\xi),
\end{equation}
where $h(x,\xi)$ is the principal symbol of $H$ regarded as a first-order operator. In fact, it is just the mean curvature form, which we denote by $\alpha$.
%Note that if one writes $H=\sum_{j=1}^q h_j(x)\frac{\partial}{\partial x_j}$ in local coordinates then $h(x,\xi)=\sum_{j=1}^q -i h_j(x)\xi_j$.

The proof of the result will follow from showing that 
\begin{equation}\label{integralzero}
 \int _0^T h(\Phi^s(x,\xi))\,ds=0,
\end{equation}
and thus the first wave invariant for $\Delta_B$ is the same as that for $\Delta_O$. But this follows from the fact that $H$ is conservative, by \eqref{meancurvature}, with a basic potential function and from the fact that the hamiltonian flow $\Phi^t$ is dual to the geodesic flow. In other words, \eqref{integralzero} is the integral of an exact basic form around a geodesic curve that is closed in the quotient. As such, its endpoints lie in the same leaf. In fact, $\alpha=d\Bigl(\ln(\frac{1}{Lvol(x)})\Bigr)$ and so is the differential of a basic function. More precisely, we have the following:

\begin{eqnarray}
\int_0^{T}h(\Phi^s(x,\xi))\,ds&=&\int_0^T g(H(\gamma(s)),\dot{\gamma}(s))\,ds\\
&=&\int_{\gamma} \alpha\\
&=&\int_0^T d(\log(\frac{1}{Lvol(\gamma(s))})\,ds=0,
\end{eqnarray}
since $\gamma(0)=\gamma(T)$ belong to the same leaf.
% insert calculation here.

On the other hand, from \cite{SU2011}, $\sigma_0(T,\mathcal{O})$ is equal to 
\begin{equation}
\Bigl[\int_{S(\tilde{Z}^T_{k_N})}\sigma(U_\mathcal{O}) d\mu_{\tilde{Z}^T_{k_N} } \Bigr]\,\tau^{\frac{e_T-1}{2}}\sqrt{d\tau},
\end{equation}
%d\mu_{\tilde{Z}^T_{k_N} }
where $\sigma(U_\mathcal{O})$ denotes the symbol of the (ordinary) wave kernel on $\mathcal{O}$, and $\tilde{Z}^T_k$ are the analogously defined sets of fixed points under the hamiltonian flow on $T^*\mathcal{O}$. Because the transverse flow on $M$ projects to the hamiltonian flow on $T^*\mathcal{O}$, $S(\tilde{Z}^T_{k_N})=N(Z^T_{k_N}).$ It is also immediate that $\mu_{\tilde{Z}^T_{k_N}}=\mu'_{Z^T_{k_N}}.$

Then, since $sub(\Delta_{\mathcal{O}})=0$, we conclude that  $\sigma_0(T,(M,\mathcal{F}))=\sigma_0(T,\mathcal{O}),$ proving the result. %\end{proof}

The proofs of the first case is straightforward. The proof of the result over the singular strata follows from the fact that hypothesis guarantees that the fixed point sets under the flow are contained in the respective singular strata, and hence depend on the underlying orbifold, $\mathcal{O}_k.$ We recall from \cite{AR2016b} that a version of equation \eqref{meancurvature} holds for each singular stratum, not just the regular one. In other words, on each foliated stratum $(\Sigma_k,\mathcal{F}_k),$ we have a basic mean curvature vector field
\begin{equation}
H_k=-\nabla \log(Lvol_k(x)),
\end{equation}
where $Lvol_k(x)$ is the volume of the leaf $L\subset \Sigma_k$ containing $x.$ 

Let $\mathcal{O}_k$ denote the orbifold represented by the leaf space $\Sigma_k/\mathcal{F}_k.$ Note that $\mathcal{O}_k$ is not necessarily closed. Nonetheless, we can still think about the wave invariants associated to closed geodesics on $\mathcal{O}_k.$

Suppose $T$ is in the length spectrum $\mathcal{RT}(M,\mathcal{F}),$ and there exist relatively closed curves of this length that are confined to one or more of the singular strata, and suppose that for the component of $Z^T$ of largest quotient dimension, the curves of that length do not leave any of the strata $\Sigma_k$. By the general wave trace formula in Theorem \ref{t:newsirftrace}, the leading term in the expansion has order $-\frac{1}{4}-\frac{e_T}{2},$ and the coefficient $\sigma_0(T,k)$ are given by \eqref{e:newsigmazero}. 

We now consider the restriction of $\sigma(U)$ to each of the strata in $Z^T_{max}.$ Let $k\in \{k_1, \dots, k_{N}\}$. We begin by taking suitable coordinates in a neighborhood $V$ of a point $w$ in $\Sigma_k^*\cap Z^T_k.$ Recall that $\Sigma_k$ has codimension at least 2 in $M$, \cite{Molino}. Let $d_k:=dim(\Sigma_k),$ $q_k:=d_k-k.$ We can take $V$ to be a tubular neighborhood around $\Sigma_k^*\cap Z^T_k$ that contains the point $w$. Let $\Sigma_k$ be locally cut out by the $n-d_k$ functions $z_{d_k+1},\cdots ,z_{n}$. Let $(\boldsymbol{x},\boldsymbol{y})$ be distinguished coordinates on $\Sigma_k$ where $\boldsymbol{x}=(x_1,\dots,x_{q_k})$ and $\boldsymbol{y}=(y_{q_k+1},\dots, y_{d_k})$. Let $(\boldsymbol{x},\boldsymbol{y},\boldsymbol{z};\boldsymbol{\xi},\boldsymbol{\eta},\boldsymbol{\zeta})$ denote the corresponding local coordinates in $T^*M.$ Then $V\cap Z^T_k$ is given locally by $\boldsymbol{z}=0,$ $\boldsymbol{\eta}=0,$ and $\boldsymbol{\zeta}=0$. 

Next, we choose a suitable basis in which to express the metric $g^{ij}$ on the fixed point sets. Choose a basis for the vector subspace $(T_{w}\Sigma_k)^\perp$ so that it is orthonormal with respect to the corresponding metric (which we denote by $g^{ij}_\perp(w)).$ Then extend this to a basis that is constant in the transverse directions by parallel transporting along transverse geodesics (recall that $\Sigma_k$ is transversely totally geodesic). Note that we have arranged it so that $\frac{\partial}{\partial x_i} \bigl(g^{ij}_\perp\bigr)=0$. We can then complete the basis of the tangent space by appending a basis of $T\Sigma_k$ that respects the splitting $T\Sigma_k=N\mathcal{F}_k\oplus T\mathcal{F}_k.$ With respect to this basis, $g^{ij}$ is blockwise diagonal: $g^{ij}=g^{ij}_T\oplus g^{ij}_L\oplus g^{ij}_\perp$ where $g_L^{ij}$ is the leafwise metric, and $g^{ij}_T$ is the transverse metric which does not depend on $\boldsymbol{y}$ since it is bundle-like with respect to the foliation of $\Sigma_k$.

Next, we examine the expression of the basic laplacian with an eye to calculating the restriction of the top order symbol and subprincipal symbol of $\Delta$ on $M$ to the fixed point sets $Z_k^T.$ First consider the formula for $\Delta$ for arbitrary variables $\boldsymbol{x}$:

\begin{equation}
\Delta= -\frac{1}{\sqrt{det(g)}}\sum_{ij}\frac{\partial}{\partial x_i}\bigl(\sqrt{det(g)}g^{ij}\frac{\partial}{\partial x_j}\bigr)
\end{equation}
Let $I_1=\{1,\dots q_k\},$ $I_2=\{q_k+1,\dots d_k\}$, and $I_3=\{d_k+1,\dots n\}.$ With respect to the coordinates previously chosen, we can split the sum over $i$ and $j$ over the three blocks: 

\begin{eqnarray}
\Delta = -\frac{1}{\sqrt{det(g)}}\Bigl[\sum_{i,j\in I_1}&&\!\!\!\!\!\!\!\!\!\!\!\!\frac{\partial}{\partial x_i}\bigl(\sqrt{det(g)}g^{ij}\frac{\partial}{\partial x_j}\bigr)\\ &+&\sum_{i,j\in I_2}\frac{\partial}{\partial y_i}\bigl(\sqrt{det(g)}g^{ij}\frac{\partial}{\partial y_j}\bigr)\\ &+&\sum_{i,j\in I_3}\frac{\partial}{\partial z_i}\bigl(\sqrt{det(g)}g^{ij}\frac{\partial}{\partial z_j})\Bigr]
\end{eqnarray}

Examining this expression with an eye to computing the principal and subprincipal parts of the symbol and restricting to $U\cap Z^T_k$, we see that summing over the indices in $I_2\cup I_3$ corresponding to $T\mathcal{F}\oplus (T\Sigma_k)^\perp$ produces terms that vanish when $\boldsymbol{\eta}$ and $\boldsymbol{\zeta}$ are set equal to zero. Suppressing these terms, and expanding $g$ blockwise with respect to a similar basis as its inverse $(g^{ij})$, then we can express $det(g)$ as the product of the determinants of the blocks $det(g_{ij}^T)det(g_{ij}^L)det(g_{ij}^\perp),$

\begin{eqnarray}
\Delta&=& \frac{-1}{\sqrt{det(g_{ij}^T)det(g_{ij}^L)det(g_{ij}^\perp)}}\sum_{i,j\in I_1}\frac{\partial}{\partial x_i}\bigl(\sqrt{det(g_{ij}^T)det(g_{ij}^L)det(g_{ij}^\perp)}g^{ij}\frac{\partial}{\partial x_j}\bigr)\nonumber\\&+&\cdots.
\end{eqnarray}
This can be re-expressed as follows
\begin{eqnarray}
\Delta&=& \frac{-1}{\sqrt{det(g_{ij}^T)}}\sum_{i,j\in I_1}\frac{\partial}{\partial x_i}\bigl(\sqrt{det(g_{ij}^T})g^{ij}\frac{\partial}{\partial x_j}\bigr)\\&-&
\frac{1}{\sqrt{det(g_{ij}^L)}}\sum_{i,j\in I_1}\frac{\partial}{\partial x_i}\bigl(\sqrt{det(g_{ij}^L})\bigr)\\&-&
\frac{1}{\sqrt{det(g_{ij}^\perp)}}\sum_{i,j\in I_1}\frac{\partial}{\partial x_i}\bigl(\sqrt{det(g_{ij}^\perp}\bigr)\bigr)+\cdots,
\end{eqnarray}
(again suppressing the terms that do not contribute to the restriction to $U\cap Z^T_k$).

From here, we see that the first of the three summations yields the principal symbol of this operator, which is the same as the principal symbol of $\Delta_k$, the induced Laplacian operator on $\mathcal{O}_k=\Sigma_k/\mathcal{F}_k$ restricted to basic functions on $\Sigma_k$. The second sum produces $-H_k,$ the mean curvature form for $(\Sigma_k,\mathcal{F}_k),$ and the last sum is zero when $\boldsymbol{z}=0$ by construction.

Thus, we see that $\sigma(U)$ restricted to $Z^T_k$ equals $sub(\Delta_k-H_k).$ By reasoning identical to that of the first case, it follows that if all the transverse geodesic curves of length $T$ remain in $\Sigma_k,$ then the basic spectrum determines the wave invariant $\sigma_0(T, \mathcal{O}_k)$ for the orbifold $\mathcal{O}_k.$ 
\end{proof}

\begin{remark}
If we have a representation of an orbifold $Q$ whose stratification $(\Sigma_k,\mathcal{F}_k)$ is such that for some $k,$ $\mathcal{O}_k$ is an orbifold singular stratum and all the curves of length $T$ are confined to that one orbifold stratum, then the sum in \eqref{e:expansionnewsirftrace} collapses and we can recover the first wave invariant of $\mathcal{O}_k.$ In this way, some representations may be more suited to capture invariants of the various orbifold strata than others.
 
\end{remark}

\end{subsection}
\end{section}
%working here 4/2/20
\begin{section}{Applications to Detecting Orbifold Singularities}\label{thing4}

Here we prove Theorem \ref{orbifoldsingularity} and Corollary \ref{corollaryorbifoldsingularity}. Consider a Riemannian orbifold quotient $(Q, g_Q)$ with representation $(M,\mathcal{F}).$ Observe that the closed geodesics that are orthogonal to the leaves on $(M,\mathcal{F})$ always descend to closed geodesics on $Q$ of the same length, and, thus, $Lspec^\perp(M,\mathcal{F})\subset Lspec(Q)$. On the other hand, if a closed geodesic on $Q$ contains an orbifold singularity, then it is possible that the geodesic does not arise from a smoothly closed geodesic on $M,$ that is orthogonal to the leaves, but rather arises from a geodesic ``lasso" at the singularity $\bar{p}\in \mathcal{O}$ where some element $\gamma$ in the local isotropy group at $\bar{p}$ corrects the angle formed by the lasso.
\begin{subsection}{The Case of a Good Orbifold}

To illustrate this more precisely, consider a very good orbifold $Q=\Gamma\setminus M$, and suppose there exists a closed (but not smoothly closed) geodesic $\bar{\phi}=\pi_Q(\phi)$ where $\phi$ is closed but not smoothly closed geodesic on $M$ that passes through an orbifold singularity at $p\in M$ with corresponding singularity in $Q$ denoted by $\bar{p}$, and local isotropy group $G_p\subset \Gamma$. Suppose $\phi(0)=\phi(\tau)$ and $\phi'(\tau)=d\gamma_*(\phi'(0))$ for some element $\gamma$ of $G_p$ acting non-trivially then $\tau\in Lspec(Q)$. If $\tau\notin Lspec(M)$ (in other words, if there is no other smoothly closed curve of length $\tau$) then $Lspec(M)$ is properly contained in $Lspec(Q)$ and the existence of such a $\tau\in Lspec(Q)\setminus Lspec(M)$ indicates the presence of non-trivial isotropy (regardless of whether or not the associated wave trace singularities cancel). We can generalize this argument to prove Theorem \ref{orbifoldsingularity}.

\end{subsection}

\begin{proof}[The proof of Theorem \ref{orbifoldsingularity}.]
More generally, consider any closed orbifold. In place of the global quotient $(M,\Gamma)$ we consider the frame bundle representation $(\mathcal{F}r(\mathcal{O}), O(q))$ with the usual metric, which we denote by $g_1.$ The main problem here is to identify the smoothly closed geodesic curves on the ambient space $M$ that are orthogonal to the leaves. The length spectrum of $\mathcal{F}r(\mathcal{O})$ contains the lengths of all smoothly closed geodesics, not just the ones that are strictly orthogonal to the leaves. 

We claim that the smoothly closed orthogonal geodesics are precisely the smoothly closed geodesics in $\mathcal{F}r(\mathcal{O})$ that are invariant under leaf-wise rescaling in $\mathcal{R}_\mu(\mathcal{F}r(\mathcal{O}), O(q),H=0)$, and hence the corresponding length spectrum is that given by equation \eqref{invariantperiods}. Let $c(t)$ be any curve on $\mathcal{F}r(\mathcal{O}),$ and let $\dot{c}(t)=(\dot{c}_\perp, \dot{c}_L)$ denote its tangent decomposed with respect to $t$ decomposed with respect to the transverse and leafwise splitting on the tangent space. One can observe by writing down the arc length functional 
\begin{equation}
\ell(c, g_\mu, I)=\int^{t_0}_0 (\|\dot{c}_\perp \|_{g_Q}^2+ \rho^2\|\dot{c}_L \|_{g_L}^2)^{1/2}\,dt
\end{equation}
 for the length of $c$ over any interval $I=[0,t_0]$ with respect to $g_\mu$ that if $\mu_1>\mu_2$ (and thus $\rho_1>\rho_2$ ) then the only curves whose lengths over $I$ remain unchanged are those for which $\dot{c}_L=0$ over $I$. It follows that the lengths of closed geodesics that are invariant under rescaling are precisely those in $Lspec(\mathcal{F}r(Q),g_\mu)$ for all $\mu$.

%If there exists a $\tau \in Lspec(Q,g_Q)\setminus L^\perp(\mathcal{F}r(\mathcal{O},O(q)),$ then $Q$ cannot be a manifold. This condition can be rephrased as follows: \\ if there exists a $\tau \in \mathcal{RT}(\mathcal{F}r(\mathcal{O}), O(q))\setminus L^\perp(\mathcal{F}r(\mathcal{O}),$ then $Q$ cannot be a manifold.

It just remains to show that $L^\perp(\mathcal{F}r(\mathcal{O}), O(q))$ can be defined as the elements of the length spectrum that are invariant under the rescaling described in the proof of Theorem \ref{infinitefamilies}. Let $\phi$ be a closed geodesic on $\mathcal{F}r(\mathcal{O}).$ Its tangent vector can be decomposed as a direct sum of its orthogonal and leaf-wise parts: $\dot{\phi}(t)=(\dot{\phi}_\perp(t), \dot{\phi}_L(t))$.
\end{proof}

Next we prove Corollary \ref{corollaryorbifoldsingularity}.

\begin{proof}[Proof of Corollary \ref{corollaryorbifoldsingularity}.] 
Consider now iterates of a closed geodesic in the orbifold $Q=\mathcal{O}$ that is closed, but not smoothly closed. Suppose we find a smallest $\tau\in Lspec(Q)\setminus L^\perp(\mathcal{F}r(\mathcal{O}), O(q)).$ Let $\bar{\phi}$ be the corresponding closed geodesic in $Q$. Then $\bar{\phi}(0)=\bar{\phi}(T)$ and by iterating we have $\bar{\phi}(0)=\bar{\phi}(jT)$ for all $j\in \mathbb{Z}.$ Hence, we have $\phi'(jT)=(d\gamma_*)^j(\phi'(0)).$ Since the isotropy is finite, for some $k$, we have $\gamma^k$ equal to the identity, and thus $k$ divides the order of $\gamma$ and the corresponding geodesic must be smoothly closed. If $j$ is such that $\dot{\phi}(jT)\not=\dot{\phi}(0)$ for all $1<j<k$, then $\gamma$ has order $k.$
\end{proof}
%\end{subsection}

\begin{subsection}{Some Remarks}
The hypotheses of Theorem \ref{orbifoldsingularity} and Corollary \ref{corollaryorbifoldsingularity} are admittedly hard to check because it is hard to determine if a given $T$ lies in the set of periods that are invariant under leaf-wise rescaling. In practice, if we are interested in checking whether or not a particular $T$ corresponds to the length of a smoothly closed curve, then we make the following observation. There exists a closed non-orthogonal geodesic of smallest length $T_{min}$ which depends on $\mu$. It is possible to choose $\mu$ so large so that $T_{min}>T+1$. Thus, the interval $[0,T+1]$ will not contain the lengths of any closed non-orthogonal geodesics for this choice of $\mu$, and thus, if $T\in Lspec(\mathcal{F}r(\mathcal{O}, O(q), g_\mu)$ then $T\in Lspec^\perp(\mathcal{F}r(\mathcal{O}), O(q))$.
\end{subsection}

\end{section}

%\section*{Acknowledgements}
%The author would like to thank Tracy Payne and Carolyn Yackel for their support during the writing of this manuscript.

%\section*{References}

%\bibliographystyle{abbrv}
%\bibliography{sandoval}
\bibliographystyle{alpha} 
%\bibliography{biglist}

\end{document}